\documentclass[12pt]{article}
\usepackage[latin1]{inputenc}

\usepackage{amsfonts}
\usepackage{amssymb}
\usepackage{amsthm}
\usepackage{amsmath}
\usepackage{mathrsfs}
\usepackage[latin1]{inputenc}
\usepackage[all]{xy}
\usepackage[left=1.20in,right=1.20in]{geometry}
\usepackage{xcolor}

\newcommand{\ent}{\mathbb Z}

\newcommand{\lac}{\mathscr{C}}

\newcommand{\fracb}[2]{\frac{\raisebox{-.7ex}{$\displaystyle #1$}}{\raisebox{-.7ex}{$\displaystyle #2$}}}

\theoremstyle{plain}
\newtheorem{teo}{Theorem}[section]
\newtheorem{prop}[teo]{Proposition}
\newtheorem{coro}[teo]{Corollary}
\newtheorem{lema}[teo]{Lemma}

\theoremstyle{definition}
\newtheorem{defi}[teo]{Definition}

\theoremstyle{remark}
\newtheorem{rem}[teo]{Remark}
\newtheorem{ejem}[teo]{Example}

\author{Jes\'us Ibarra$^*$, Alberto G. Raggi-C\'ardenas\footnote{CCM, UNAM, Morelia, Mexico.}$\ $ and Nadia Romero\footnote{DEMAT, UGTO, Guanajuato, Mexico. 
}}
\title{The additive completion of the biset category}

\date{ }

\begin{document}

\maketitle

\begin{abstract}
Let $R$ be a commutative unital ring. We construct a category $\lac_R$ of \textit{fractions} $X/G$, where $G$ is a finite group and $X$ is a finite  $G$-set, and with morphisms given by $R$-linear combinations of spans of bisets. 
This category is an additive, symmetric monoidal and self-dual category, 
with a Krull-Schmidt decomposition for objects.  We show that $\lac_R$  is equivalent to the additive completion of the biset category and that  the category of biset functors over $R$ is equivalent to the cate\-gory of $R$-linear functors from $\lac_R$ to $R$-Mod. We also show that the restriction of one of these functors to a certain subcategory of $\lac_R$ is a fused Mackey functor.\\

\end{abstract}

\section*{Introduction}

The aim of this paper is to give an explicit description of the additive completion of the biset category and to show some of its properties. 

The additive completion of the biset category will be equivalent to the category $\lac_R$, the objects of which are couples $(G,\, X)$ where $G$ is a finite group and $X$ is a finite \textcolor{black}{(left)} $G$-set. These objects will be written as fractions $\frac{X}{G}$
for convenience (or as $X/G$ in case of limited space). One of the nice features of this notation is that $\lac_R$ is an additive symmetric monoidal category with addition and tensor product given by
\begin{displaymath}
\frac{X}{G}\oplus \frac{Y}{H}=\frac{(X\times H)\sqcup (G\times Y)}{G\times H}\quad\textrm{and}\quad \frac{X}{G}\otimes \frac{Y}{H}=\frac{X\times Y}{G\times H}.
\end{displaymath}
The arrows and the composition in $\lac_R$ will be explained in detail in the next section.

We will see that objects of the form $\{\bullet \}/G$ are indecomposable in $\lac_R$ and that every object in $\lac_R$ can be written as a sum of this kind of objects in a unique way, up to isomorphism. This gives a Krull-Schmidt decomposition for objects in $\lac_R$, although this does not make $\lac_R$ a Krull-Schmidt category, in the sense of the formal definition, because the endomorphism ring of $\{\bullet\}/G$ in $\lac_R$ is isomorphic to the double Burnside ring $RB(G,\,G)$, which in general is not a local ring. We will also show that  $\lac_R$ is a self-dual category and that it has a Mackey decomposition for arrows. 

In the last section we consider $R$-linear functors from $\lac_R$ to $R$-Mod and their relation with Mackey and biset functors. First we will prove that for a finite group $G$ we have a functor from the Burnside category (the category having as objects the isomorphism classes of finite $G$-sets and as arrows the spans of $G$-sets) to $\lac_R$. This functor is injective in objects but it is not full, nor faithful. Nevertheless, we will see that by pre-composition with this functor, an $R$-linear functor from $\lac_R$ to $R$-Mod gives a \textit{fused Mackey functor}, as defined in \cite{bham}. On the other hand, we prove that the category of biset functors is equivalent to the category of $R$-linear functors from $\lac_R$ to $R$-Mod. As an example, we see that the Burnside biset functor extends to the functor from $\lac_R$ to $R$-Mod which sends an object $X/G$ to the Burnside group $B(X)$.
 
This paper originates in results appearing in the PhD thesis of the first author \cite{tachosth}. A related, more general, construction to our category $\lac_R$ is given by Nakaoka \cite{naka}.

\section{Preliminaries}

In this section $G$, $H$ and $K$ will be  finite groups.

\subsection*{About $G$-sets}

We will write $B(G)$ for the Burnside group (ring) of finite $G$-sets.
 
Recall, for example from Section 1.6 in Bouc \cite{boucbook}, that given a $G$-set $X$, a couple $(Y,\, \varphi)$ is called \textit{a $G$-set over $X$} if $Y$ is a $G$-set and $\varphi$ is a morphism of $G$-sets from $Y$ to $X$. A morphism from $(Y,\, \varphi)$ to $(Y',\, \varphi')$ is a morphism of $G$-sets $f:Y\rightarrow Y'$ such that $\varphi' f=\varphi$. Whit this data one can define the category of $G$-sets over $X$, denoted by $G$-set$\downarrow_X$. The Grothendieck group of the category $G$-set$\downarrow_X$, for relations given by decomposition into disjoint union, is denoted by $B(X)$. It is easy to see that $B(X)$ has an additive structure and it is thus called the  \textit{Burnside group of the $G$-set $X$}. By Proposition 2.4.2 in \cite{boucbook}, if $X$ is a transitive $G$-set, isomorphic to a set of left cosets $G/H$, then $B(X)$ coincides with $B(H)$, the Burnside group of the group $H$, this will also be a consequence of Lemma \ref{punto}. This construction can be endowed with a structure of Green functor, for more details see Section 2.4 in \cite{boucbook}. In particular, $B(X)$ has a multiplicative structure, given explicitly by the pullback in the category $G$-set, extended by bilinearity.

The \textit{Burnside category} (see for example Lindner \cite{lindner}), denoted by $Span(G$-set$)$ or $Span_R(G$-set$)$ if we are taking coefficients in a commutative unital ring $R$, is the category of spans in the category $G$-set, namely: objects in $Span(G$-set$)$ are finite $G$-sets, and the set of morphisms from a $G$-set $X$ to a $G$-set $Y$, denoted by $B^G(Y\times X)$, is defined as follows. The set $Y\times X$ has a natural structure of $G$-set, so we consider the elements of the form $(S,\, \beta\times\alpha)$ in $B(Y\times X)$,  where 
\begin{displaymath}
\xymatrix{
&S\ar[dl]_{\beta}\ar[dr]^{\alpha}&\\
Y & &X
}
\end{displaymath}
is a span in $G$-set and $(\beta\times\alpha)(s)=(\beta(s),\,\alpha(s))$.
Then $B^G(Y\times X)$ is generated by the equivalence classes of these elements 
under the relation which makes $Y\leftarrow S\rightarrow X$ equivalent to \mbox{$Y\leftarrow U\rightarrow X$} if there is an isomorphism of $G$-sets between $S$ and $U$ which commutes with the projection maps to $X$ and $Y$. Composition of morphisms in $Span(G$-set$)$ is given by the pullback in the category $G$-set and extended by bilinearity. We will denote by $\langle S,\beta,\alpha\rangle$ the equivalence class of $(S,\, \beta\times\alpha)$ in $B^G(Y\times X)$.

In $Span_R(G$-set$)$, morphisms are given by $RB^G(Y\times X):=R\otimes_{\ent}B^G(Y\times X)$.

Lindner shows in \cite{lindner} that Mackey functors can be defined using the Burnside category.

\begin{defi}
Let $R$ be a commutative unital ring and $G$ be a finite group. A Mackey functor for $G$ over $R$ is an $R$-linear functor from $Span_R(G$-set$)$ to the category of $R$-modules, $R$-Mod. 

The category of Mackey functors, with arrows given by natural transformations, is denoted by
\begin{displaymath}
Fun_R(Span_R(G\textrm{-set}),\, R\textrm{-Mod}).
\end{displaymath}
\end{defi}

\subsection*{About bisets}

For a commutative unital ring $R$, the \textit{biset category} with coefficients in $R$, denoted by $\Omega_R$ is the category having as objects the class of all finite groups, and as morphisms from a group $G$ to a group $H$ the Burnside group $RB(H\times G)=R\otimes_{\ent} B(H\times G)$. There is a one-to-one correspondence between $(H\times G)$-sets and $(H,\, G)$-bisets, which are sets endowed with compatibles left $H$-action and right $G$-action, given by the following relation
\begin{displaymath}
(h,\, g)x=hxg^{-1},
\end{displaymath}
for $x$ in $X$, an $(H\times G)$-set ($(H,\, G)$-biset). So, basic elements in $RB(H\times G)$ will be usually referred to as $(H,\, G)$-bisets, and $RB(H\times G)$ will also be denoted by $RB(H,\, G)$. Given $U$ an $(H,\, G)$-biset and $V$ a $(K,\, H)$-biset, the composition of $V$ and $U$, denoted by $V\times_H U$ is the set of $H$-orbits on the cartesian product $V\times U$, where the right action of $H$ is defined by
\begin{displaymath}
 \forall (v,\, u)\in V\times U,\, \forall h\in H,\ (v,\, u)\cdot h=(v\cdot h,\, h^{-1}\cdot u). 
\end{displaymath}
This product has a natural structure of $(K,\, G)$-biset. Composition in $\Omega_{R}$ is given by extending by bilinearity this product. The identity element in $RB(G,\,G)$ is (the class of) the $(G,\,G)$-biset $G$.

\begin{defi}
Let $R$ be a commutative unital ring. A biset functor is an $R$-linear functor from $\Omega_R$ to the category of $R$-modules, $R$-Mod. 

The category of biset functors, with arrows given by natural transformations, is denoted by
\begin{displaymath}
Fun_R(\Omega_R,\, R\textrm{-Mod}).
\end{displaymath}
\end{defi}

Basic properties of the biset category and of biset functors can be found in Bouc \cite{biset}.

\section{Construction of the category}

The category we will define in this section has some similarities with both, the Burnside category and the biset category. 

Throughout this section, $R$ will be a commutative unital ring. 

\begin{defi}
\label{lacat}
The category $\lac_R$ is defined in the following way:
\begin{itemize}
\item[$\bullet$] Objects are the couples $(G,\, X)$, where $G$ is a finite group and $X$ is a finite $G$-set. The couple $(G,\, X)$ will be denoted as a fraction $\frac{X}{G}$, or as $X/G$ in case of limited space. The reason for this notation will become clear in the next theorem. 
\item[$\bullet$] Let $X/G$ and $Y/H$ be objects in $\lac_R$. Then $Y$ is an $(H\times G)$-set through the projection map from $H\times G$ to $H$, and $X$ is an $(H\times G)$-set in a similar way. 
So we set 
\begin{displaymath}
Hom_{\lac_R}\left(\frac{X}{G},\, \frac{Y}{H}\right)=RB^{H\times G}(Y\times X),
\end{displaymath}
which we will write from now on as $RB^{H,\, G}(Y,\, X)$, to denote a set of arrows.

\item[$\bullet$] Composition of morphisms is given as follows: Let $\langle U,\beta,\alpha\rangle$ be a generator of $RB^{H,\, G}(Y,\, X)$ and $\langle V,\delta, \gamma\rangle$ be a generator of $RB^{K,\, H}(Z,\, Y)$. Then the composition $\langle V,\delta,\gamma\rangle\circ\langle U,\beta,\alpha\rangle$ is defined as $\langle V\hat{\times}_H U,\hat{\delta},\hat{\alpha}\rangle$ in $RB^{K,\, G}(Z,\, X)$, where $V\hat{\times}_HU$ and the morphisms $\hat{\delta}$ and $\hat{\alpha}$ are defined after the diagram. 
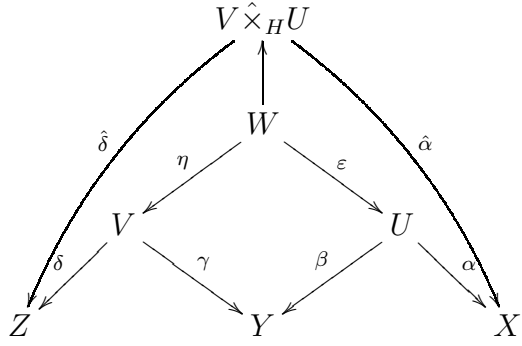
\[
\xymatrix{
&& V\hat{\times}_HU\ar@/_1pc/[dddll]_{\hat\delta} \ar@/^1pc/[dddrr]^{\hat\alpha} &&\\
&& W\ar[u]\ar[rd]^\varepsilon\ar[ld]_{\eta} &&\\
& V\ar[rd]^{\gamma}\ar[ld]_\delta & & U \ar[rd]^{\alpha}\ar[ld]_{\beta}& \\
Z && Y && X 
}
\]
The pullback $W$ of $\gamma$ and $\beta$ is the subset of $V\times U$ of couples $(v,\, u)$ such that $\gamma (v)=\beta(u)$, together with the projections $\eta$ and $\epsilon$. If we consider $V$ as a $(K,\, H)$-biset and $U$ as an $(H,\, G)$-biset, then $V\times U$ has a right action of $H$ as described in the previous section, and $W$ is stable under this action, thus $V\hat{\times}_HU$ is the set of $H$-orbits of $W$, that is
\begin{displaymath}
V\hat{\times}_HU=\{[v,\, u]\in V\times_HU\mid \gamma(v)=\beta(u)\},
\end{displaymath}
which has a structure of $(K\times G)$-set as seen in the previous section. The maps $\hat{\delta}[v,\, u]=\delta(v)$ and $\hat{\alpha}[v,\, u]=\alpha(u)$ are well defined by the next remark. It is also easy to see that they are morphisms of $(K\times G)$-sets.

In the next proposition we prove that this product is well defined, and that it can be extended by bilinearity to give a composition in $\lac_R$.
 
\item[$\bullet$] Let $X/G$ be an object in $\lac_R$. Recall that $X$ can be seen as a $(G\times G)$-set through the projections on the first or second variable, $\pi_1$ and $\pi_2$ respectively. In the set of morphisms  $RB^{G,\, G}(X,\, X)$, we are considering the $X$ on the left as a $(G\times G)$-set through $\pi_1$ and the one on the right through $\pi_2$. 

The identity element in $RB^{G,\, G}(X,\, X)$ is $\langle G\times X, \underline{1},\overline{1}\rangle$, where $G\times X$ is seen as a $(G\times G)$-set through the action $(a,\, b)(g,\, x)=(agb^{-1},\, ax)$ and the maps
 $\underline{1},\overline{1}:G\times X\to X$ are defined by $\overline{1}(g,x)=g^{-1}x$ and $\underline{1}(g,x)=x$. They are easily seen to be morphisms of $(G\times G)$-sets.
\end{itemize}
\end{defi}

We verify in the next proposition and theorem that $\lac_R$ is an additive monoidal category.

\begin{rem}
Let $X/G$ and $Y/H$ be objects in $\lac_R$. A triple $\langle U,\beta,\alpha\rangle$, with $U$ an $(H\times G)$-set, is in $RB^{H,\,G}(Y,\, X)$, 
if and only if $\beta:U\rightarrow Y$ and \mbox{$\alpha:U\rightarrow X$} are functions that satisfy $\beta ((h,\, g)u)=h\beta (u)$, for all $g\in G$ and \mbox{$\alpha ((h,\, g)u)=g\alpha (u)$} for all $h\in H$.
This is because the action of $G$ in $Y$ is trivial, and so is the action of $H$ in $X$. In particular if we see $U$ as an $(H,\, G)$-biset, these conditions translate as $\beta(hug)=h\beta(u)$ for all $g\in G$ and $\alpha (hug)=g^{-1}\alpha(u)$ for all $h\in H$.
\end{rem}

To verify that $\lac_R$ is indeed a category, we prove first some properties of the composition.

\begin{prop}
Consider the composition of generating morphisms in $\lac_R$ as in Definition \ref{lacat}. This composition is well defined, associative and bilinear. Also, the identity morphism for the object $X/G$ is given by $\langle G\times X,\, \underline{1},\, \overline{1} \rangle$.
\end{prop}
\begin{proof}
We begin by showing that the composition is well defined. Let $\langle U,\, \beta,\,\alpha\rangle$, $\langle U',\,\beta',\,\alpha'\rangle$ be in $RB^{H,\, G}(Y,\, X)$, and $\langle V,\,\delta,\, \gamma\rangle$, $\langle V',\, \delta',\,\gamma'\rangle$ be in $RB^{K,\, H}(Z,\, Y)$ with  $(U,\,\beta,\,\alpha)\backsimeq (U',\,\beta',\,\alpha')$ and $(V,\,\delta,\,\gamma)\backsimeq (V',\,\delta',\,\gamma')$. We have the diagram
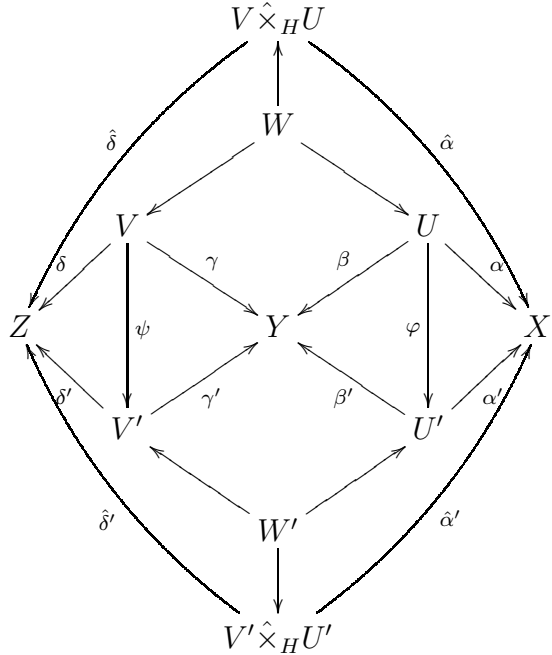
\[
\xymatrix{
&& V\hat{\times}_HU\ar@/_1pc/[dddll]_{\hat\delta} \ar@/^1pc/[dddrr]^{\hat\alpha} &&\\
&& W\ar[u]\ar[rd]\ar[ld] &&\\
& V\ar[dd]^{\psi}\ar[rd]^{\gamma}\ar[ld]_{\delta} & & U\ar[dd]_{\varphi}\ar[rd]^{\alpha}\ar[ld]_{\beta}& \\
Z && Y && X \\
 & V'\ar[ru]_{\!\gamma'}\ar[lu]^{\delta'\!\!} & & U' \ar[ru]_{\!\!\alpha'}\ar[lu]^{\beta'\!}& \\
 && W'\ar[d]\ar[ru]\ar[lu] && \\
 && V'\hat{\times}_HU'\ar@/^1pc/[uuull]^{\hat\delta'} \ar@/_1pc/[uuurr]_{\hat\alpha'} &&\\
}
\]
where
\begin{displaymath}
W= \left\{(v,u)\in V\times U\mid\gamma(v)=\beta(u)\right\},\
W'= \left\{(v',u')\in V'\times U'\mid=\gamma'(v')=\beta'(u')\right\},
\end{displaymath}
\begin{displaymath}
V\hat{\times}_HU= \left\{[v,u]\in V\times_H U\mid\gamma(v)=\beta(u)\right\}\textrm{ and}
\end{displaymath}
\begin{displaymath}
V'\hat{\times}_HU'= \left\{[v',u']\in V'\times_H U'\mid\gamma'(v')=\beta'(u')\right\}.
\end{displaymath}
Suppose that $\varphi: U\rightarrow U'$ is an isomorphism of $(H,\,G)$-bisets such that $\beta=\beta'\varphi$ and $\alpha=\alpha'\varphi$, and that $\psi:V\rightarrow V'$ is an isomorphism of $(K,\, H)$-bisets such that $\delta=\delta'\psi$ and $\gamma=\gamma'\psi$. We define
\begin{displaymath}
\rho:V\hat{\times}_HU\rightarrow V'\hat{\times}_HU'\textrm{ s.t. }[v,\, u]\mapsto [\psi(v),\, \varphi(u)].
\end{displaymath}
Since $\psi$ and $\varphi$ respect the action of $H$ in $V$ and $U$, respectively, it is easy to see that $\rho$ is well defined. Also, since $\delta=\delta'\psi$ and $\gamma=\gamma'\psi$, then $[\psi(v),\, \varphi(u)]$ is indeed in $V'\hat{\times}_HU'$. It is also clear that $\rho$ is a morphisms of $(K,\, G)$-bisets, and that we can defined an inverse, from $V'\hat{\times}_HU'$ to $V\hat{\times}_HU$, satisfying analogous properties, i.e. $\rho$ is an isomorphism of $(K,\,G)$-bisets. Finally, to see that $\hat{\delta}=\hat{\delta}'\rho$ and $\hat{\alpha}=\hat{\alpha}'\rho$, it suffices to follow the definitions of $\hat{\delta}$, $\hat{\alpha}$, given in the previous definition.

We now prove associativity. Let
\begin{displaymath}
\xymatrix{\fracb{X_1}{G}\ar[rr]^{\langle U,\,\beta_1,\,\alpha_1 \rangle}&&\fracb{X_2}{H}\ar[rr]^{\langle V,\,\beta_2,\,\alpha_2\rangle}&& \fracb{X_3}{K}\ar[rr]^{\langle W,\,\beta_3,\,\alpha_3\rangle}&& \fracb{X_4}{L}}
\end{displaymath}
be generating morphisms in $\lac_R$. Then we have the following diagram for $\langle W,\,\beta_3,\,\alpha_3\rangle\circ(\langle V,\,\beta_2,\,\alpha_2\rangle\circ\langle U,\,\beta_1,\,\alpha_1 \rangle)$:
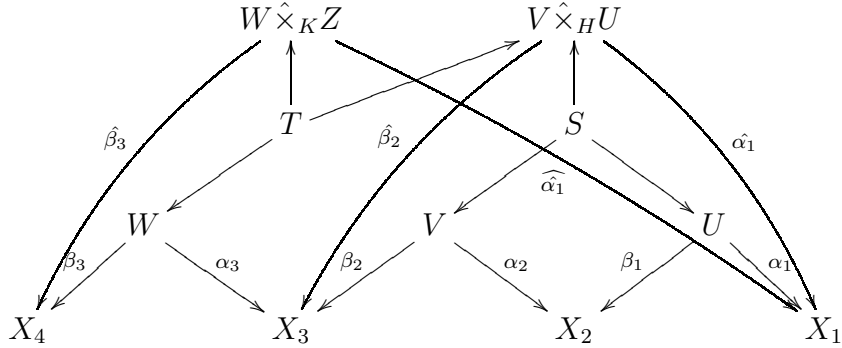
\[
\xymatrix{
& &W\hat{\times}_KZ\ar@/_1pc/[llddd]_{\hat{\beta_3}} \ar@/^/[rrrrddd]_{\widehat{\hat{\alpha_1}}} &&  V\hat{\times}_HU\ar@/_1pc/[llddd]_{\hat{\beta_2}}\ar@/^1pc/[rrddd]^{\hat{\alpha_1}} && \\
&& T\ar[u]\ar[ld]\ar[rru] && S\ar[u]\ar[ld]\ar[rd] &\\
& W\ar[ld]_{\beta_3\!\!}\ar[rd]^{\!\alpha_3} && V\ar[ld]_{\beta_2\!}\ar[rd]^{\!\alpha_2} &&U\ar[ld]_{\beta_1\!}\ar[rd]^{\!\!\alpha_1} &\\
X_4 && X_3 && X_2 && X_1    
}
\]
where $Z=V\hat{\times}_HU$. That is 
\begin{displaymath}
W\hat{\times}_K(V\hat{\times}_HU)=\{[w,\,[v,\, u]]\in W\times_K(V\hat{\times}_HU)\mid \alpha_3(w)=\hat{\beta_2}([v,\, u])\},
\end{displaymath}
where $[v,\, u]$ in $[w,\,[v,\, u]]$ satisfies $\alpha_2(v)=\beta_1(u)$, which means that
\begin{displaymath}
W\hat{\times}_K(V\hat{\times}_HU)=\{[w,\,[v,\, u]]\in W\times_K(V\times_HU)\mid \alpha_3(w)=\beta_2(v),\ \alpha_2(v)=\beta_1(u)\}.
\end{displaymath}
The corresponding morphisms are given by 
\begin{displaymath}
\hat{\beta_3}([w,\,[v,\, u]])=\beta_3(w)\quad\textrm{and}\quad\widehat{\hat{\alpha_1}}([w,\,[v,\, u]])=\hat{\alpha_1}([v,\, u])=\alpha_1(u).
\end{displaymath}
On the other hand, by a similar reasoning, we have
\begin{displaymath}
(W\hat{\times}_KV)\hat{\times}_HU=\{[[w,\,v],\, u]\in (W\times_KV)\times_HU\mid \alpha_3(w)=\beta_2(v),\ \alpha_2(v)=\beta_1(u)\},
\end{displaymath}
with corresponding morphisms
\begin{displaymath}
\widehat{\hat{\beta_3}}([[w,\,v],\, u])=\beta_3(w)\quad\textrm{and}\quad\tilde{\alpha_1}([[w,\, v],\, u])=\alpha_1(u).
\end{displaymath}
Then, it is not hard to see that the isomorphism $W\times_K(V\times_HU)\cong (W\times_KV)\times_HU$ induces an isomorphism $W\hat{\times}_K(V\hat{\times}_HU)\cong(W\hat{\times}_KV)\hat{\times}_HU$ which commutes with $\hat{\beta_3}$, $\widehat{\hat{\alpha_1}}$ and $\widehat{\hat{\beta_3}}$, $\tilde{\alpha_1}$. This means that we can write $W\hat{\times}_KV\hat{\times}_HU$ and identify it with
\begin{displaymath}
\{[w,\,v,\, u]\in W\times_KV\times_HU\mid \alpha_3(w)=\beta_2(v),\ \alpha_2(v)=\beta_1(u)\}.
\end{displaymath}
Next we show that this composition is bilinear. Recall that addition of morphisms is given by disjoint union. Let $\langle U_1,\,\beta_1,\,\alpha_1\rangle$ and $\langle U_2,\,\beta_2,\,\alpha_2\rangle$ be generating arrows from $X/G$ to $Y/H$ and $\langle V,\,\delta,\,\gamma\rangle$ from $Y/H$ to $K/Z$. We will prove that
\begin{displaymath}
\langle V,\,\delta,\,\gamma\rangle\circ(\langle U_1,\,\beta_1,\,\alpha_1\rangle +\langle U_2,\,\beta_2,\,\alpha_2\rangle)= \langle V,\,\delta,\,\gamma\rangle\circ\langle U_1,\,\beta_1,\,\alpha_1\rangle+\langle V,\,\delta,\,\gamma\rangle\circ\langle U_2,\,\beta_2,\,\alpha_2\rangle.
\end{displaymath}
For $\langle V,\,\delta,\,\gamma\rangle\circ(\langle U_1,\,\beta_1,\,\alpha_1\rangle +\langle U_2,\,\beta_2,\,\alpha_2\rangle)$, we have the diagram
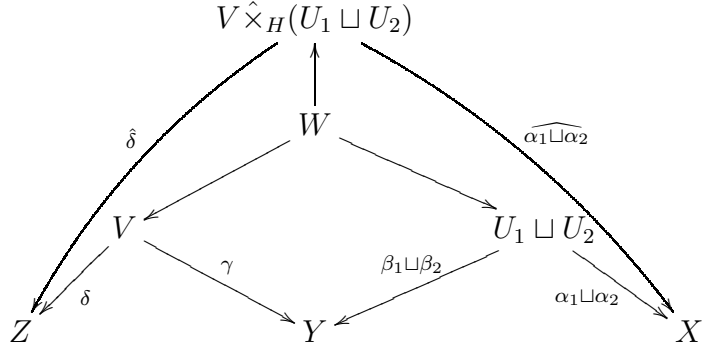
\[
\xymatrix{
&& V\hat{\times}_H(U_1\sqcup U_2)\ar@/_1pc/[dddll]_{\hat\delta} \ar@/^1pc/[dddrr]^{\widehat{\alpha_1\sqcup \alpha_2}} &&\\
&& W\ar[u]\ar[rd]\ar[ld] &&\\
& V\ar[rd]^{\gamma}\ar[ld]^\delta & & U_1\sqcup U_2 \ar[rd]_{\alpha_1\sqcup\alpha_2}\ar[ld]_{\beta_1\sqcup \beta_2}& \\
Z && Y && X 
}
\]
where
\begin{displaymath}
V\hat{\times}_H(U_1\sqcup U_2)=\left\{[v,\,u]\in V\times_H \left(U_1\sqcup U_2\right)\mid\gamma(v)=(\beta_1\sqcup \beta_2)(u)\right\},
\end{displaymath}
with $\hat{\delta}([v,\, u])=\delta(v)$ and $\widehat{\alpha_1\sqcup\alpha_2}([v,\, u])=(\alpha_1\sqcup\alpha_2)(u)$.

On the other hand, $\langle V,\,\delta,\,\gamma\rangle\circ\langle U_i,\,\beta_i,\,\alpha_i\rangle$ is equal to $\langle V\hat{\times}_HU_i,\,\hat{\delta}_i,\, \hat{\alpha}_i\rangle$, for $i=1,\,2$, where
\begin{displaymath}
V\hat{\times}_HU_i=\left\{[v,u]\in V\times_H U_i\mid\gamma(v)=\beta_i(u)\right\}, 
\end{displaymath}
with $\hat{\delta}_i([v,\,u])=\delta_i(v)$ and $\hat{\alpha}_i([v,\, u])=\alpha_i(u)$.
 
With these descriptions it is clear that $V\hat{\times}_H(U_1\sqcup U_2)=(V\hat{\times}_HU_i)\sqcup (V\hat{\times}_HU_i)$ and that $\hat{\delta}=\hat{\delta}_1\sqcup \hat{\delta}_2$ and $\widehat{\alpha_1\sqcup\alpha_2}=\hat{\alpha}_1\sqcup\hat{\alpha}_2$.

Analogous arguments show linearity on the other variable.

Finally, to prove that $\langle G\times X,\, \underline{1},\, \overline{1} \rangle$ is the identity morphism for $X/G$, let $\langle U,\, \beta,\, \alpha\rangle$ be a generating morphism in $RB^{H,\, G}(Y,\, X)$. The composition $\langle U,\, \beta,\, \alpha\rangle\circ\langle G\times X,\, \underline{1},\, \overline{1}\rangle$ is given by
\begin{displaymath}
\langle U\hat{\times}_G(G\times X),\, \hat{\beta},\, \hat{\overline{1}}\rangle,
\end{displaymath}
where $\hat{\beta}([u,\, (g,\, x)])=\beta (u)$ and $\hat{\overline{1}}([u,\, (g,\, x)])=\overline{1}((g,\, x))=g^{-1}x$.

We define 
\begin{displaymath}
\varphi:U\hat{\times}_G(G\times X)\rightarrow U,\quad\textrm{as}\quad [u,\, (g,\, x)]\mapsto ug.
\end{displaymath}
It is easy to see that it is well defined and is a morphism of $(H,\, G)$-bisets. Observe that an element $[u,\, (g,\,x)]$ in $U\hat{\times}_G(G\times X)$ satisfies $\alpha(u)=\underline{1}(g,\, x)=x$, so it is not hard to see that the inverse of $\varphi$ is given by 
\begin{displaymath}
\psi:U\rightarrow U\hat{\times}_G(G\times X),\quad u\mapsto [u,\,(1,\, \alpha(u))].
\end{displaymath}
Finally, given the definitions of $\hat{\beta}$ and $\hat{\overline{1}}$, it is straightforward to see that $\alpha\varphi=\hat{\overline{1}}$ and $\beta\varphi=\hat{\beta}$.

Similar arguments show that $\langle G\times X,\, \underline{1},\, \overline{1}\rangle$ is the identity when composed with a generating morphism $\langle V,\,\gamma,\,\delta\rangle$ in $RB^{G,\, H}(X,\, Y)$.
\end{proof}

\begin{teo}
Letting $\lac_R$ be as in Definition \ref{lacat}, we have that it is an additive, symmetric monoidal and self-dual category, with addition and tensor product given by: For $X/G$ and $Y/H$ objects in $\lac_R$, 
\begin{displaymath}
\frac{X}{G}\oplus\frac{Y}{H}=\frac{(X\times H)\sqcup (G\times Y)}{G\times H}\quad\textrm{and}\quad \frac{X}{G}\otimes \frac{Y}{H}=\frac{X\times Y}{G\times H},
\end{displaymath}
where $X\times H$, $G\times Y$ and $X\times Y$ are seen as $(G\times H)$-sets through the projections on $G$ and $H$.
\end{teo}
\begin{proof}
By the previous proposition, $\lac_R$ is a pre-additive category, we first prove that $\lac_R$ has finite coproducts.

For the zero-ary case, it is easy to see that $\emptyset/\{1\}$ is an initial object. Observe that for any group $G$ we have
\begin{displaymath}
\frac{\emptyset}{\{1\}}\cong \frac{\emptyset}{G}
\end{displaymath}
in $\lac_R$.

For the bi-coproducts, we begin by finding the  corresponding morphisms from the objects $X/G$ and $Y/H$ to \mbox{$((X\times H)\sqcup (G\times Y))/(G\times H)$}. For the one from $X/G$, that is the one in $RB^{G\times H,\, G}\big((X\times H)\sqcup (G\times Y),\, X\big)$, we take $\langle G\times H\times X,\,\underline{i_{X}},\overline{1_X}\rangle$, where $G\times H\times X$ is regarded as a $(G\times H\times G)$-set through the action 
\begin{displaymath}
(g_1,\, h_1,\, g_2)(g,\, h,\, x)=(g_1gg_2^{-1},\, h_1h,\, g_1x),
\end{displaymath}
with morphisms $\underline{i_X}(g,\,h,\,x)=(x,\, h)$ and $\overline{1_X}(g,\,h,\,x)=g^{-1}x$. It is easy to see that $\underline{i_X}$ and $\overline{1_X}$ are morphisms of $(G\times H\times G)$-sets. Similarly, for $Y/H$, we consider $\langle G\times H\times Y,\, \underline{i_Y},\, \overline{1_Y}\rangle$ in $RB^{G\times H,\, H}\big((X\times H)\sqcup (G\times Y),\, Y\big)$, where $G\times H\times Y$ has a $(G\times H\times H)$-action given by
\begin{displaymath}
(g_1,\, h_1,\, h_2)(g,\, h,\, y)=(g_1g,\, h_1hh_2^{-1},\, h_1y),
\end{displaymath}
and morphisms $\underline{i_Y}(g,\,h,\,y)=(g,\, y)$ and $\overline{1_Y}(g,\,h,\,y)=h^{-1}y$.

Now, suppose we have an object $Z/K$ and morphisms 
\begin{displaymath}
f:\frac{X}{G}\longrightarrow \frac{Z}{K}\quad \textrm{and}\quad t:\frac{Y}{H}\longrightarrow \frac{Z}{K}.
\end{displaymath}
 Since composition in $\lac_R$ is bilinear, we can suppose that $f$ is equal to a generating element $\langle U,\, \beta,\, \alpha\rangle$ in $RB^{K,\, G}(Z,\, X)$ and $t=\langle W,\, \eta,\, \varepsilon \rangle$ in $RB^{K,\, H}(Z,\, Y)$. Finding a morphism $\varphi$ in $RB^{K,\, G\times H}\big(Z,\,(X\times H)\sqcup (G\times Y)\big)$ such that
\begin{displaymath}
\varphi\circ\langle G\times H\times X,\,\underline{i_{X}},\overline{1_X}\rangle \backsimeq\langle U,\beta,\alpha\rangle\quad\textrm{and}\quad\varphi\circ\langle G\times H\times Y,\,\underline{i_{Y}},\overline{1_Y}\rangle \backsimeq\langle W,\eta,\,\varepsilon\rangle,
\end{displaymath}
will show that $((X\times H)\sqcup (G\times Y))/G\times H$ is indeed a coproduct. We define
\begin{displaymath}
 \varphi=\langle U\times H,\,\underline{\beta},\, \alpha^+\rangle +\langle G\times W,\, \underline{\eta},\, \varepsilon^+ \rangle
\end{displaymath}
where $U\times H$ and $G\times W$ are seen as $(K\times G\times H)$-sets with the actions 
\begin{displaymath}
(k,\,g,\, h_1)(u,\, h)=((k,\,g)u,\,hh_1^{-1})\quad \textrm{and}\quad
(k,\,g_1,\, h)(g,\, w)=(gg_1^{-1},\,(k,\,h)w)
\end{displaymath}
and we have 
\begin{displaymath}
\underline{\beta}(u,\, h)=\beta(u),\ \alpha^+(u,\,h)=(\alpha(u),\, h^{-1}),\ \underline{\eta}(g,\, w)=\eta(w) \textrm{ and } \varepsilon^+(g,\, w)=(g^{-1},\, \varepsilon(w)),
\end{displaymath}
for $(u,\, h)\in U\times H$ and $(g,\, w)\in G\times W$. It is easy to see that they are all morphisms of $(K\times G\times H)$-sets.

For $X/G$ we have that $\varphi\circ\langle G\times H\times X,\,\underline{i_{X}},\overline{1_X}\rangle$ is equal to
\begin{displaymath}
\langle (U\times H)\hat{\times}_{G\times H}(G\times H\times X),\, \widehat{\underline{\beta}},\, \widehat{\overline{1_X}}\rangle,
\end{displaymath}
since in $(X\times H)\sqcup (G\times Y)$ there are no elements satisfying $\varepsilon^+(g,\,w)=\underline{i_X}(g,\, h,\, x)$.

We have the diagram
\[
\xymatrix{
&&& T\ar@/_1pc/[dddlll]_{\widehat{\underline{\beta}}} \ar@/^1pc/[dddrrr]^{\widehat{\overline{1_X}}} &&\\
&&& P\ar[u]\ar[rd]\ar[ld] &&\\
&& U\times H\ar[rd]^{\alpha^+}\ar[lld]_{\underline{\beta}\!\!\!} & & G\times H\times X \ar[rrd]^{\!\!\!\!\overline{1_X}}\ar[ld]_{\underline{i_X}}& \\
Z &&& (X\times H)\sqcup (G\times Y) &&& X 
}
\]
where $T=(U\times H)\hat{\times}_{G\times H}(G\times H\times X)$.
An element here, $q=[(u,\,h),\, (g_1,\, h_1,\,x)]$ satisfies $(\alpha(u),\, h^{-1})=(x,\, h_1)$, so we can write $q=[(u,\,h),\, (g_1,\, h^{-1},\, \alpha (u))]$, which is in turn equal to $[(u,\, 1),\, (g_1,\,1,\, \alpha(u))]$.

We prove now that
\begin{displaymath}
\langle (U\times H)\hat{\times}_{G\times H}(G\times H\times X),\, \widehat{\underline{\beta}},\, \widehat{\overline{1_X}}\rangle\backsimeq\langle U,\beta,\alpha\rangle.
\end{displaymath}

We define $f:U\rightarrow (U\times H)\hat{\times}_{G\times H}(G\times H\times X)$ as $f(u)=[(u,\,1),(1,\,1,\ \alpha(u))]$. 
It is easy to see that $f$ is injective. To see that it is surjective, observe that
\begin{eqnarray*}
[(u,\, 1),\, (g_1,\,1,\, \alpha(u))]& =&[(u,\, 1),\, (g_1,\, 1)(1,\,1,\, g_1^{-1}\alpha(u))]\\
 &=& [(u,\, 1)(g_1,\,1),\, (1,\,1,\, g_1^{-1}\alpha(u))]\\
 &=& [((1,\,g_1^{-1})u,\, 1),\, (1,\,1,\, \alpha((1,\,g_1^{-1})u))],
\end{eqnarray*}
which is equal to $f((1,\,g_1^{-1})u)$.

Next we see that $f$ is a morphism of $(K\times G)$-sets. Recall that the action of $K\times G$ in $(U\times H)\hat{\times}_{G\times H}(G\times H\times X)$ is given through $K\times \{1\}$ in $U$ and through $\{1\}\times G$
 in $G\times H\times X$, and so $(k,\, g_1)f(u)$ is equal to
\begin{displaymath}
[((k,\,1)u,\, 1),\, (g_1^{-1},\,1,\, \alpha(u))].
\end{displaymath}
On the other hand, $f((k,\, g_1)u)$ is equal to
\begin{displaymath}
[((k,\, g_1)u,\, 1),\, (1,\, 1,\, \alpha((k,\, g_1)u)],
\end{displaymath}
which is equal to
\begin{eqnarray*}
[((k,\, g_1)u,\, 1),\, (1,\, 1,\, g_1\alpha(u))]& = & [((k,\, 1)u,\, 1)(g_1^{-1},\,1),\, (g_1,\, 1)(g_1^{-1},\, 1,\, \alpha(u)]\\
& = & [((k,\, 1)u,\, 1),\, (g_1^{-1},\, 1,\, \alpha(u)].
\end{eqnarray*}

Finally, it is easy to see that $\widehat{\underline{\beta}}f=\beta$ and $\widehat{\overline{1_X}}f=\alpha$.

Analogous arguments show the corresponding results for $Y/H$.

To see that it is a symmetric monoidal category, with unit element $\{\bullet\}/\{1\}$, it suffices to use the next lemma, and to identify
\begin{displaymath}
\left(\frac{X}{G}\otimes \frac{Y}{H}\right)\otimes \frac{Z}{K}\quad \textrm{and}\quad 
\frac{X}{G}\otimes \left(\frac{Y}{H}\otimes \frac{Z}{K}\right)
\end{displaymath}
with the object
\begin{displaymath}
\frac{X\times Y\times Z}{G\times H\times K},
\end{displaymath}
defined in the obvious way.

To see that $\lac_R$ is self-dual, let $F:\lac_R\to\lac_R^{op}$ be the functor defined as the identity on objects and that sends a generating morphism $\langle U,\beta,\alpha\rangle$ in $RB^{H,\, G}(Y,\, X)$ to
$\langle U^{op},\alpha,\beta\rangle$,
where $U^{op}$ is the opposite $(G\times H)$-set. It is straightforward to see that $F$ is an equivalence of categories.
\end{proof}

\begin{lema}
\label{isosi}
Let $G$ and $H$ be finite groups, $X$ a finite $G$-set and $Y$ be a finite $H$-set. If $G$ and $H$ are isomorphic groups, through an isomorphism $f:G\rightarrow H$, and  there exists a bijection $t:X\rightarrow Y$ that satisfies $t(gx)=f(g)t(x)$, then $X/G$ and $Y/H$ are isomorphic objects in $\lac_R$.
\end{lema}
\begin{proof}
Take $\langle G\times X,\, \underline{1_X},\, \overline{t}\rangle$ in $RB^{G,\, H}(X,\, Y)$ and $\langle H\times Y,\, \underline{1_Y},\, \overline{t^{-1}}\rangle$ in $RB^{H,\, G}(Y,\, X)$, where $G\times X$ is a $(G\times H)$-set and $H\times Y$ is an $(H\times G)$-set through the following actions
\begin{displaymath}
(g,\, h)(g_1,\, x)=(gg_1f^{-1}(h^{-1}),\, gx)\quad\textrm{and}\quad(h,\, g)(h_1,\, y)=(hh_1f(g^{-1}),\, hy),
\end{displaymath}
and we have
\begin{displaymath}
\underline{1_X}(g,\, x)=x,\quad \underline{1_Y}(h,\, y)=y,\quad \overline{t}(g,\, x)=t(g^{-1}x) \ \textrm{ and }\ \overline{t^{-1}}(h,\, y)=t^{-1}(h^{-1}y).
\end{displaymath}
We verify only that $\langle G\times X,\, \underline{1_X},\, \overline{t}\rangle\circ\langle H\times Y,\, \underline{1_Y},\, \overline{t^{-1}}\rangle$ is equal to $\langle G\times X,\, \underline{1},\, \overline{1}\rangle$, since to prove that the other composition gives the identity is analogous.
The composition $\langle G\times X,\, \underline{1_X},\, \overline{t}\rangle\circ\langle H\times Y,\, \underline{1_Y},\, \overline{t^{-1}}\rangle$ is equal to
$\langle (G\times X)\hat{\times}_H(H\times Y),\, \widehat{\underline{1_X}},\, \widehat{\overline{t^{-1}}} \rangle$, with
\begin{displaymath}
(G\times X)\hat{\times}_H(H\times Y)=\{[(g,\,x),\,(h,\, y)]\mid f(g^{-1})t(x)=y\}
\end{displaymath}
and
\begin{displaymath}
\widehat{\underline{1_X}}([(g,\,x),\,(h,\, y)])=x\quad \textrm{and}\quad\widehat{\overline{t^{-1}}}([(g,\,x),\,(h,\, y)])=f^{-1}(h^{-1})t^{-1}(y).
\end{displaymath}
It is easy to observe that an element $[(g,\,x),\,(h,\, y)]$ in $(G\times X)\hat{\times}_H(H\times Y)$ is equal to
\begin{displaymath}
[(gf^{-1}(h),\, x),\, (1,\, (h^{-1}f(g^{-1})t(x))],
\end{displaymath}
so we can define $\alpha:(G\times X)\hat{\times}_H(H\times Y)\rightarrow G\times X$ as $\alpha([(g,\,x),\,(h,\, y)])=(gf^{-1}(h),\, x)$. It is not hard to see that $\alpha$ is a well defined surjective morphism of $(G\times G)$-sets. To see that $\alpha$ is injective, suppose that $[(g_1,\,x_1),\,(h_1,\, y_1)]$ and $[(g_2,\,x_2),\,(h_2,\, y_2)]$ are two elements in $(G\times X)\hat{\times}_H(H\times Y)$ such that $(g_1f^{-1}(h_1),\,x_1)=(g_2f^{-1}(h_2),\,x_2)$. This clearly implies
\begin{displaymath}
[(g_1f^{-1}(h_1),\, x_1),\, (1,\, (h_1^{-1}f(g_1^{-1})t(x_1))]=[(g_2f^{-1}(h_2),\, x_2),\, (1,\, (h_2^{-1}f(g_2^{-1})t(x_2))].
\end{displaymath}
Finally, we clearly have $\underline{1}\alpha=\widehat{\underline{1_X}}$ and $\overline{1}\alpha=\widehat{\overline{t^{-1}}}$.
\end{proof}

\section{Some properties of $\lac_R$}

We begin by identifying some isomorphic objects in $\lac_R$, and we finish the section by showing that $\lac_R$ is the additive completion of the biset category.

\begin{lema}
\label{isosvarios}
Let $G$ and $H$ be finite groups and $X$ and $X'$ be finite $G$-sets. In $\lac_R$ we have
\begin{displaymath}
\textrm{a)}\ \frac{X}{G}\cong \frac{X\times H}{G\times H}\quad \textrm{and}\quad
\textrm{b)}\ \frac{X}{G}\oplus \frac{X'}{G}\cong \frac{X\sqcup X'}{G}.
\end{displaymath}
\end{lema}
\begin{proof}
The proof of a) comes from the fact that $\emptyset/H$ is isomorphic to $\emptyset/\{1\}$ and so
\begin{displaymath}
\frac{X}{G}\cong \frac{X}{G}\oplus \frac{\emptyset}{H}=\frac{X\times H}{G\times H}.
\end{displaymath}

To prove b), observe that by Lemma \ref{isosi}, we have
\begin{displaymath}
\frac{X}{G}\oplus \frac{X'}{G}=\frac{(X\times G)\sqcup (G\times X')}{G\times G}\cong \frac{(X\sqcup X')\times G}{G\times G},
\end{displaymath}
which, by the previous point, is isomorphic to $(X\sqcup X')/G$.
\end{proof}

\begin{lema}
\label{punto}
If $H\leqslant G$, consider $G/H$, the set of left cosets of $H$ in $G$.  Then in $\lac_R$,
\begin{displaymath}
\frac{G/H}{G}\cong \frac{\{\bullet\}}{H}.
\end{displaymath}
\end{lema}
\begin{proof}
Define a morphism from $(G/H)/G$ to $\{\bullet\}/H$ as $\langle G,\, e, \pi_1\rangle$, where $G$ is seen as a $(H\times G)$-set through the action $(h,\, g_1)g=hgg_1^{-1}$, the morphism $e:G\to \{\bullet\}$ is the obvious one and $\pi_{1}:G\to G/H$ maps $g$ to $g^{-1}H$. They clearly are morphisms of $(H\times G)$-sets. Its inverse will be the morphism $\langle G,\, \pi,\, e\rangle$, where $G$ is a $(G\times H)$-set with action $(g_1,\,h)g=g_1gh^{-1}$, $\pi$ is the projection map on $H$ and $e$ is defined as before.

We will only verify that $\langle G,\, \pi,\, e \rangle\circ \langle G,\, e, \pi_{1}\rangle\backsimeq \langle G\times (G/H),\,\underline{1},\, \overline{1}\rangle$, the identity morphism for $(G/H)/G$, since to prove that the other composition is equivalent to the identity morphism for $\{\bullet\}/H$ is straightforward.

We have the following diagram
\[
\xymatrix{
&& G\hat{\times}_HG\ar@/_1pc/[dddll]_{\hat\pi} \ar@/^1pc/[dddrr]^{\hat{\pi}_{1}} &&\\
&& W\ar[u]\ar[rd]\ar[ld] &&\\
& G\ar[rd]^{e}\ar[ld]_{\pi\!\!\!} & & G \ar[rd]^{\!\!\!\!\pi_{1}}\ar[ld]_{e}& \\
G/H && \{\bullet\} && G/H
}
\]
Observe that in this case $G\hat{\times}_HG=G\times_HG$.

To see that this composition is equivalent to the identity, let $f:G\times_HG\to G\times (G/H)$ be defined by $f([a,\,b])=(ab,\,aH)$. We prove that $f$ is an isomorphism of $(G\times G)$-sets such that $\overline{1} f=\hat{\pi}_{1}$ and $\underline{1} f=\hat{\pi}$. It is clear that $f$ is well defined. To see that it is a morphism of $(G\times G)$-sets, observe that
\begin{displaymath}
 f((r,\,s)[a,\,b])=f([ra,\,bs^{-1}])=(rabs^{-1},raH)=(r,\,s)f([a,b]).
\end{displaymath}
It is easy to see that the inverse of $f$ is the morphism $t:G\times (G/H)\to G\times_HG$ defined by $t(c,\, dH)=[d,\, d^{-1}c]$. Finally, we have that
\begin{displaymath}
\overline{1}f([a,b])=b^{-1}a^{-1}aH=b^{-1}H=\hat{\pi}_1([a,b])
\end{displaymath}
and that $\underline{1}f([a,b])=aH=\hat{\pi}([a,b])$. This finishes the proof.
\end{proof}

Some consequences of this lemma are the following.

\begin{rem}
\label{hom}
Any of both, point a) of Lemma \ref{isosvarios} or the previous lemma, implies that for any group $G$, we have 
\begin{displaymath}
\frac{G}{G}\cong \frac{\{\bullet\}}{\{1\}}.
\end{displaymath}

Also, observe that if $X$ is a transitive $G$-set, isomorphic to a set of left cosets $G/H$, then the previous lemma implies that the Burnside group of $X$, as defined in Section 1, is isomorphic to
\begin{displaymath}
B^{G,\, \{1\}}(X,\, \{\bullet\})\cong B^{H,\, \{1\}}(\{\bullet\},\, \{\bullet\})\cong B(H).
\end{displaymath}
\end{rem}

\begin{rem}
\label{deco}
Point b) of Lemma \ref{isosvarios} and the previous lemma imply that any $X/G$ in $\lac_R$ can be decomposed as
\begin{displaymath}
\frac{X}{G}\cong \bigoplus_{i=1}^n\frac{\{\bullet\}}{G_i},
\end{displaymath}
if $X$ decomposes in $G$-orbits as $X=\bigsqcup_{i=1}^nO_i$ and the stabilizer of an element in $O_i$ is $G_i$. 
\end{rem}

To prove the following theorem we will use the functor we define next.

\begin{defi}
Let $X/G$ be an element in $\lac_R$ and $K$ be a finite group. We define
\begin{displaymath}
\overline{RB}^{G,\, K}(X,\, \{\bullet\})=\overline{Hom}_{\lac_R}\left(\frac{\{\bullet\}}{K},\, \frac{X}{G}\right)=Hom_{\lac_R}\left(\frac{\{\bullet\}}{K},\, \frac{X}{G}\right)/\mathcal{F},
\end{displaymath}
where $\mathcal{F}$ is the subgroup generated by elements of the form $u\circ v$, with
\begin{displaymath}
 u\in Hom_{\lac_R}\left(\frac{\{\bullet\}}{K},\,\frac{\{\bullet\}}{H}\right),\ v\in Hom_{\lac_R}\left(\frac{\{\bullet\}}{H},\, \frac{X}{G}\right)
\end{displaymath} 
and $H$ runs through groups such that  $|H|<|K|$.
\end{defi}

As before, observe that 
\begin{displaymath}
Hom_{\lac_R}\left(\frac{\{\bullet\}}{K},\,\frac{\{\bullet\}}{H}\right)=RB^{H,\, K}(\{\bullet\},\,\{\bullet\})
\end{displaymath}
is isomorphic to the Burnside group $RB(H,\, K)$.

It is not hard to verify that for $K$ a given group, $\overline{RB}^{\,\_\,,\, K}(\,\_\,,\, \{\bullet\})$
defines an additive functor from $\lac_R$ to the category of abelian groups. 

\begin{lema}
Let $X/G$ be an element in $\lac_R$ and $K$ be a finite group. Then 
\begin{displaymath}
\overline{RB}^{G,\,K}(X,\,\{\bullet\})\neq 0
\end{displaymath}
if and only if there exists $x\in X$ such that $K$ is a subquotient of  $Stab_G(x)$.
\end{lema}
\begin{proof}
By Remark \ref{deco}, $\overline{RB}^{G,\,K}(X,\,\{\bullet\})\neq 0$ if and only if, there exists $x\in X$ such that for $G_x=Stab_G(x)$, we have 
\begin{displaymath}
\overline{RB}^{G_x,\,K}(\{\bullet\},\,\{\bullet\})\neq 0.
\end{displaymath}
This happens if and only if there exists $U$, a transitive $(G_x,\, K)$-biset, which does not factor, with the composition of bisets, through groups of order smaller than $|K|$. Then, the butterfly decomposition for bisets, Lemma 2.3.26 in \cite{biset}, gives that we have this if and only if $K$ is  a subquotient of $G_x$.
\end{proof}

\begin{teo}
\label{iso}
Let $G$ and $H$ be finite groups, $X$ be a finite $G$-set and $Y$ be a finite $H$-set. Suppose that as a $G$-set, $X$ decomposes in orbits as $\bigsqcup_{i=1}^nO_i$, with the stabilizer of an element in $O_i$ being $G_i$, and as an $H$-set, $Y$ decomposes in orbits as $\bigsqcup_{j=1}^mT_j$, with the stabilizer of an element in $T_j$ being $H_j$. If $X/G$ is isomorphic to $Y/H$ in $\lac_R$, then $n=m$ and there exists $\sigma\in S_n$ such that $G_i\cong H_{\sigma(i)}$ as groups.
\end{teo}
\begin{proof}
Notice that if $X$ or $Y$ is the empty set, the result is trivial, so we can suppose that $X$ is not empty. 

Since $X/G$ and $Y/H$ are isomorphic, for any group $L$ we have
\begin{displaymath}
\overline{RB}^{G,\, L}(X,\, \{\bullet\})\cong \overline{RB}^{H,\, L}(Y,\, \{\bullet\}).
\end{displaymath}
We observe next, using Remark \ref{deco}, that $\overline{RB}^{G,\, G_i}(X,\, \{\bullet\})\neq 0$ for any $G_i$, because 
\begin{displaymath}
\overline{RB}^{G_i,\, G_i}(\{\bullet\},\, \{\bullet\})\neq 0,
\end{displaymath}
for it contains the $(G_i,\,G_i)$-biset $G_i$. Hence we can take $K$ of maximal order such that $\overline{RB}^{G,\, K}(X,\, \{\bullet\})\neq 0$. Clearly $K$ is also maximal with the property $\overline{RB}^{H,\, K}(Y,\, \{\bullet\})\neq 0$. Then there are $i$ and $j$ such that $K$ is isomorphic to a subquotient of $G_i$ and isomorphic to a subquotient of $H_j$. We show that $G_i\cong K\cong H_j$. If $K\ncong G_i$, then $G_i$ has order larger than $K$ and $\overline{RB}^{G,\, G_i}(X,\, \{\bullet\})\neq 0$, a contradiction. The same argument shows $K\cong H_j$. Then we have 
\begin{displaymath}
\frac{\{\bullet\}}{K}\oplus \frac{X'}{G}\cong \frac{\{\bullet\}}{K}\oplus \frac{Y'}{H}
\end{displaymath}
where $X'=X\setminus O_i$, with $O_i$ the orbit corresponding to $G_i$  and $Y'=Y\setminus T_j$, with $T_j$ the orbit corresponding to $H_j$. This isomorphism implies that for any group $L$ we have
\begin{displaymath}
Hom_{\lac_R}\left(\frac{\{\bullet\}}{L},\,\frac{\{\bullet\}}{K}\oplus \frac{X'}{G}\right)\cong Hom_{\lac_R}\left(\frac{\{\bullet\}}{L},\,\frac{\{\bullet\}}{K}\oplus \frac{Y'}{H}\right),
\end{displaymath}
which in turn implies
\begin{displaymath}
RB(K,\, L)\oplus RB^{G,\,L}(X',\,\{\bullet\})\cong RB(K,\, L)\oplus RB^{H,\,L}(Y',\,\{\bullet\}).
\end{displaymath}
Hence, given that in both sides of this isomorphism we have free $R$-modules of finite rank, we have $RB^{G,\,L}(X',\,\{\bullet\})\cong RB^{H,\,L}(Y',\,\{\bullet\})$, and so
\begin{displaymath}
\overline{RB}^{G,\, L}(X',\, \{\bullet\})\cong \overline{RB}^{H,\, L}(Y',\, \{\bullet\}).
\end{displaymath}
From this point forward, the argument we used above works exactly the same for $X'$, $Y'$ and the corresponding (remaining) stabilizers $G_i$ and $H_j$. Proceeding by induction on the number of orbits, we obtain the result.
\end{proof}

\begin{coro}
\label{comoKSR}
If $G_i$, for $1\leq i\leq n$, and $H_j$, for $1\leq j\leq m$, are groups such that
\begin{displaymath}
\bigoplus_{i=1}^n\frac{\{\bullet\}}{G_i}\cong \bigoplus_{j=1}^m\frac{\{\bullet\}}{H_j}
\end{displaymath}
in $\lac_R$, then $m=n$ \textcolor{black}{and} there exists $\sigma\in S_n$ such that $G_i\cong H_{\sigma(i)}$ as groups.

In particular, the decomposition given in Remark \ref{deco} is unique up to isomorphism on the $G$-stabilizers of $X$.
\end{coro}
\begin{proof}
It is not hard to see that the sum on the left-hand side will be an object on $\lac_R$ of the form $Z/A$, where $A=G_1\times \cdots\times G_n$ and $Z$ is an $A$-set composed by $n$ orbits with stabilizers isomorphic to $G_1,\, \ldots, G_n$. Likewise, the right-hand side of the isomorphism will be an object of the form $W/B$, where $B=H_1\times \cdots\times H_m$ and $W$ is a $B$-set composed by $m$ orbits with stabilizers isomorphic to $H_1,\, \ldots, H_m$. The result follows from Theorem \ref{iso}.
\end{proof}

\begin{coro}
Let $X/G$,  $Y/H$ and $Z/L$ be objects in $\lac_R$. We have a cancellation law for the sum in $\lac_R$, that is, if
\begin{displaymath}
\frac{X}{G}\oplus\frac{Y}{H}\,\cong\, \frac{X}{G}\oplus\frac{Z}{L},
\end{displaymath}
then $Y/H$ is isomorphic to $Z/L$.
\end{coro}
\begin{proof}
Since $X/G$ can be decomposed as
$\bigoplus_{i=1}^n\{\bullet\}/G_i$,
we can assume, proceeding by induction, that we have the following isomorphism
\begin{displaymath}
\frac{\{\bullet\}}{G}\oplus\frac{Y}{H}\,\cong\, \frac{\{\bullet\}}{G}\oplus\frac{Z}{L}.
\end{displaymath}
The result follows from the previous corollary.
\end{proof}

\begin{rem}
These results imply that $\lac_R$ is almost a Krull-Schmidt category. Indeed, Theorem \ref{deco} provides, for each object in $\lac_R$, a decomposition into indecomposable objects which is unique up to isomorphism. Nevertheless, for a group $L$, the endomorphism ring of $\{\bullet\}/L$, the double Burnside ring $RB(L,\,L)$, is not in general a local ring.
\end{rem}

\begin{prop}[Mackey formula]
Let $\langle U,\,\beta,\, \alpha\rangle$ in $RB^{H,\,G}(Y,\, X)$ and $\langle V,\, \delta,\,\gamma\rangle$ in $RB^{K,\, H}(Z,\, Y)$ be generating elements such that $U=(H\times G)/D$ is a transitive $(H\times G)$-set and $V=(K\times H)/E$ is a transitive $(K\times H)$-set. Suppose also that $\beta (D)=\gamma (E)$. Letting $H'=Stab_H(\beta(D))$, we have
\begin{displaymath}
\langle V,\, \delta,\,\gamma\rangle\circ\langle U,\,\beta,\, \alpha\rangle=\sum\limits_{h\in [A\backslash H'/ B]} \langle (K\times G) /(E*^{(h,\,1)}D),\,\hat{\delta}_h,\,\hat{\alpha}_h\rangle
\end{displaymath}where
\begin{displaymath}
E*^{(h,\,1)}D=\left\{(k,\,g)\in K\times G\mid \exists h'\in H \textrm{ s. t. } (k,\,h')\in E,\, (h',\,g)\in {}^{(h,\,1)} D\right\},
\end{displaymath}
the group $A$ is the projection of $E$ on $H$ and $B$ is the projection of $D$ on $H$. Letting $T_h=E*^{(h,\,1)}D$, we have the maps $\hat{\delta}_h((k,\,g)T_h)=k\delta(E)$ and $\hat{\alpha}_h((k,\,g)T_h)=g\alpha(D)$.
\end{prop}
\begin{proof}
Recall that the composition $\langle V,\, \delta,\,\gamma\rangle\circ\langle U,\,\beta,\, \alpha\rangle$ is equal to
$\langle V\hat{\times}_HU,\, \hat{\delta},\, \hat{\alpha}\rangle$,
where $V\hat{\times}_HU=\{[v,\,u]\mid \gamma(v)=\beta (u)\}$ and
\begin{displaymath}
\hat{\delta}([(k,\, h_1)E,\, (h_2,\, g)D])=\delta((k,\,h_1)E)=k\delta (E)\textrm{ and}
\end{displaymath}
\begin{displaymath}
\hat{\alpha}([(k,\, h_1)E,\, (h_2,\, g)D])=\alpha((h_2,\,g)D)=g\alpha (D).
\end{displaymath}
Observe that if $h$ is in $A$, there exists $k\in K$ such that $(k,\,h)$ is in $E$, hence
\begin{displaymath}
h\gamma (E)=\gamma ((k,\,h)E)=\gamma (E),
\end{displaymath}
and so, since $\gamma (E)=\beta (D)$, we have $A\leqslant H'$. In the same way we see that $B\leqslant H'$. Also, a similar argument shows that an element $[(k,\, h_1)E,\, (h_2,\, g)D]$ in $V\hat{\times}_HU$ satisfies that $h_1^{-1}h_2$ is in $H'$.
On the other hand, from Lemma 2.3.24 in \cite{biset}, we know that
\begin{displaymath}
V\times_H U\cong \bigsqcup \limits_{h\in [A\backslash H/ B]} (K\times G) /(E*^{(h,\,1)}D).
\end{displaymath}
In the proof of this, the orbits of the action of $K\times G$ on $V\times_HU$ are seen to be in bijection with $[A\backslash H/ B]$ through the map that sends the orbit of an element $[(k,\, h_1)E,\, (h_2,\,g)D]$ to $Ah_1^{-1}h_2B$. This implies that
\begin{displaymath}
V\hat{\times}_H U\cong \bigsqcup \limits_{h\in [A\backslash H'/ B]} (K\times G) /(E*^{(h,\,1)}D)
\end{displaymath}
as $(K\times G)$-sets.

Finally, since the stabilizer of an element in $V\hat{\times}_HU$, under the action of $K\times G$, is of the form $T_h$ with $h$ in $[A\backslash H'/B]$, we see that $\hat{\delta}$ can be decomposed as $\bigsqcup_{h\in[A\backslash H'/B]}\hat{\delta}_h$, and $\hat{\alpha}$ as $\bigsqcup_{h\in[A\backslash H'/B]}\hat{\alpha}_h$.
\end{proof}

\begin{teo}
\label{comp}
The category $\lac_R$ is the additive completion of the biset category.
\end{teo}
\begin{proof}
We define a functor $F:\Omega_R\rightarrow \lac_R$, by sending a group $G$ to the object $\{\bullet\}/G$ and the isomorphism class of a $(G,\, H)$-biset $U$  to the class of 
\begin{displaymath}
\{\bullet\}\leftarrow U\rightarrow\{\bullet\}
\end{displaymath}
 in $RB^{G,\, H}(\{\bullet\},\,\{\bullet\})$. As we have already noticed, $RB^{G,\, H}(\{\bullet\},\,\{\bullet\})$ is isomorphic to $RB(G, H)$,
then it is clear that $F$ is an $R$-linear faithful functor. 
 
The additive completion of $\Omega_R$, see Chapter VIII of Mac Lane \cite{maclane} for instance, is the category $Add (\Omega_R)$ consisting of $n$-tuples, for $n$ a non-negative integer,   of objects of $\Omega_R$, with arrows given by matrices of arrows in $\Omega_R$. There is a canonical functor $A:\Omega_R\rightarrow Add(\Omega_R)$ sending a group $G$ to the $1$-tuple $(G)$ and the class of a $(G,\, H)$-biset $U$ to the $1\times 1$ matrix $(U)$. We have then a unique functor $S:Add(\Omega_R)\rightarrow \lac_R$ such that $SA=F$. It is easy to see that $S$ is a full, faithful and dense functor, hence it is an equivalence of categories.
\end{proof}

\section{Functors from $\lac_R$ to $R$-Mod}

If $R$ is a commutative unital ring, we denote by
\begin{displaymath}
Fun_R(\lac_R,\, R\textrm{-Mod}),
\end{displaymath}
the category of $R$-linear functors from $\lac_R$ to $R$-Mod, with arrows given by natural transformations.

\textcolor{black}{ 
We see first the relation of these functors with Mackey functors. We will see that for a given group $G$, the restriction of a functor in $Fun_R(\lac_R,\, R\textrm{-Mod})$ to a certain subcategory of $\lac_R$ gives a \textit{fused Mackey functor for $G$}. Fused Mackey functors were introduced by Bouc in \cite{bham}, but they  had been previously considered, from a different point of view, by Hambleton, Taylor and Williams in \cite{hamtaywil}, where they are called \textit{conjugation invariant Mackey functors}. They are defined in \cite{bham} as $R$-linear functors from a category of spans $\mathcal{F}$ to $R$-Mod. The category $\mathcal{F}$ has for objects the finite $G$-sets, and the arrows from a given $G$-set $X$ to a $G$-set $Y$ is the quotient of $B^G(Y\times X)$ by the relation that identifies the classes of two spans 
\begin{displaymath}
\xymatrix{Y& U\ar[l]_\beta\ar[r]^\alpha & X }\quad\xymatrix{Y& U'\ar[l]_{\beta'}\ar[r]^{\alpha'} & X }
\end{displaymath}
if there exist an isomorphism  of $G$-sets $f:U\rightarrow U'$ and a morphism of $G$-sets $t:U\rightarrow G^c$, that satisfy
\begin{displaymath}
\beta'f=\beta \quad \textrm{and}\quad \alpha'f=t\cdot\alpha,
\end{displaymath}
where $G^c$ is the $G$-set $G$ with action given by conjugation and $t\cdot\alpha:U\rightarrow X$ sends $u$ to $t(u)\alpha (u)$. These are precisely the conditions we obtain in the next result.} 

\begin{prop}
Let $R$ be a commutative ring with unit and $G$ be a fixed finite group. There is a functor
\begin{displaymath}
A:Span_R(G\textrm{-set})\longrightarrow \lac_R
\end{displaymath}
\textcolor{black}{ 
such that for a functor $F$ in $Fun_R(\lac_R,\, R\textrm{-Mod})$, the composition $F\circ A$ is a fused Mackey functor for $G$.}
\end{prop}
\begin{proof}
The functor $A$ sends a $G$-set $X$ to the object $X/G$ in $\lac_R$, and the class of a span 
\begin{displaymath}
\xymatrix{Y& T\ar[l]_\beta\ar[r]^\alpha & X }
\end{displaymath}
to the arrow $\langle G\times T,\, \underline{\beta},\, \overline{\alpha} \rangle$, where $G\times T$ is a $(G\times G)$-set with the action 
\begin{displaymath}
(g_1,\, g_2)(g,\, t)=(g_1gg_2^{-1},\, g_1t),
\end{displaymath}
and the morphisms
\begin{displaymath}
\underline{\alpha}(g,\, t)=g^{-1}\alpha(t)\quad\textrm{and}\quad \overline{\beta}(g,\,t)=\beta(t),
\end{displaymath}
which are easily seen to be morphisms of $(G\times G)$-sets. Also, it is easy to see that $A$ is well defined on morphisms and sends the identity morphism of a $G$-set $X$, the class of the span $X\leftarrow X\rightarrow X$, to the identity $\langle G\times X,\, \underline{1},\, \overline{1} \rangle$ of $X/G$ in $\lac_R$. So, to see that $A$ is a functor it suffices to see that it preserves the composition. Let
\begin{displaymath}
\xymatrix{Z& S\ar[l]_\delta\ar[r]^\gamma &  Y}\quad\textrm{and}\quad \xymatrix{Y& T\ar[l]_\beta\ar[r]^\alpha & X}
\end{displaymath}
be two spans. Applying $A$ to these spans we obtain $\langle G\times S,\, \underline{\delta},\, \overline{\gamma} \rangle$ and $\langle G\times T,\, \underline{\beta},\, \overline{\alpha}\rangle$. Composition in $\lac_R$ is given by
$\langle U,\,\hat{\underline{\delta}},\,\hat{\overline{\alpha}} \rangle$, where
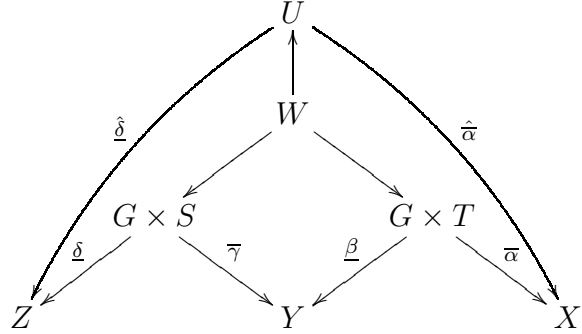
\[
\xymatrix{
&& U\ar@/_1pc/[dddll]_{\hat{\underline{\delta}}} \ar@/^1pc/[dddrr]^{\hat{\overline{\alpha}}}&&\\
&& W\ar[u]\ar[rd]\ar[ld] &&\\
& G\times S\ar[rd]^{\overline{\gamma}}\ar[ld]_{\underline{\delta}} & & G\times T \ar[rd]^{\overline{\alpha}}\ar[ld]_{\underline{\beta}}& \\
Z && Y && X
}
\]
with $U=(G\times S)\hat{\times}_G(G\times T)$. Thus elements in this set are of the form $[(g_1,\, s),\, (g_2,\,t)]$ and they satisfy $g_1^{-1}\gamma(s)=\beta(t)$. Also we have
\begin{displaymath}
\hat{\underline{\delta}}([(g_1,\, s),\, (g_2,\,t)])=\delta(s)\quad\textrm{and}\quad
\hat{\overline{\alpha}}([(g_1,\, s),\, (g_2,\,t)])=g_2^{-1}\alpha(t).
\end{displaymath}
 On the other hand, composition in $Span_R(G\textrm{-set})$ is given by the pullback
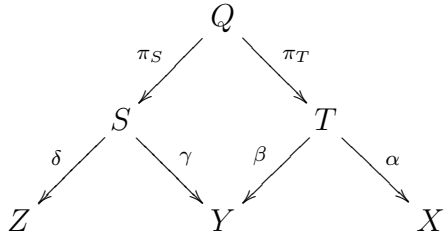
\[
\xymatrix{
&& Q \ar[rd]^{\pi_T}\ar[ld]_{\pi_S}&&\\
& S\ar[ld]_{\delta}\ar[rd]^{\gamma} & &  T\ar[ld]_{\beta}\ar[rd]^{\alpha} &\\
Z & & Y & & X
}
\]
where $Q$ is the set of couples $(s,\, t)$ such that $\gamma(s)=\beta(t)$, with a diagonal $G$-action. From this we obtain that the image of the composition under $A$ is $\langle G\times Q,\,\underline{\delta\pi_S}, \,\overline{\alpha\pi_T}\rangle$. 

We define a morphism $\varphi$ form $U$ to $G\times Q$ by
\begin{displaymath}
\varphi([(g_1,\, s),\, (g_2,\,t)])=(g_1g_2,\, (s,\, g_1t)).
\end{displaymath}
It is straightforward to see that $\varphi$ is well defined, that it has image in $G\times Q$ and that it is a morphism of $(G\times G)$-sets. Also, it is easy to see that $\varphi$ is surjective, to see that it is injective, observe that
\begin{displaymath}
[(g_1,\, s),\, (g_2,\,t)]=[(1,\, s),\, (g_1g_2,\,g_1t)],
\end{displaymath} 
so if  we have $(g_1g_2,\, (s,\, g_1t))=(g'_1g'_2,\, (s',\, g'_1t'))$, then
\begin{displaymath}
[(g_1,\, s),\, (g_2,\,t)]=[(g'_1,\, s'),\, (g'_2,\,t')].
\end{displaymath}
Finally, clearly we have $\underline{\delta\pi_S}\varphi=\hat{\underline{\delta}}$ and $\overline{\alpha\pi_T}\varphi=\hat{\overline{\alpha}}$.

\textcolor{black}{To see that by pre-composition with $A$ we obtain a fused Mackey functor, we show first that $A$ is not faithful (clearly it is not full either).} To see this, consider two spans of $G$-sets
\begin{displaymath}
\xymatrix{Y& U\ar[l]_\beta\ar[r]^\alpha & X }\quad\xymatrix{Y& U'\ar[l]_{\beta'}\ar[r]^{\alpha'} & X }
\end{displaymath}
and their corresponding images under $A$
\begin{displaymath}
\xymatrix{Y& G\times U\ar[l]_{\underline{\beta}}\ar[r]^{\overline{\alpha}} & X }\quad\xymatrix{Y& G\times U'\ar[l]_{\underline{\beta'}}\ar[r]^{\overline{\alpha'}} & X }.
\end{displaymath}
If there exists a morphisms $\varphi :G\times U\rightarrow G\times U'$ of $(G\times G)$-sets, then
\begin{displaymath}
\varphi (g,\, u)=\varphi ((1,\, g^{-1})(1,\, u))=(1,\, g^{-1})\varphi (1,\, u).
\end{displaymath}
This means that $\varphi$ is determined by its action in $\{1\}\times U$ and we can write
\begin{displaymath}
\varphi (g,\, u)=(1,\, g^{-1})(t(u),\, f(u))=(t(u)g,\, f(u)),
\end{displaymath}
for some functions $t:U\rightarrow G$ and $f:U\rightarrow U'$. Since $\varphi$ is a morphism of $(G\times G)$-sets, we also have that
\begin{displaymath}
\varphi ((k,\, h) (g,\, u))=\varphi (kgh^{-1},\, ku)=(t(ku)kgh^{-1},\, f(ku))
\end{displaymath}
is equal to $(k,\, h)(t(u)g,\, f(u))$, which is equal to $(kt(u)gh^{-1},\, kf(u))$. Hence, $f$ is a morphisms of $G$-sets and $t:U\rightarrow G^c$ is also a morphism of $G$-sets. It is not hard to verify that $\varphi$ is an isomorphism of $(G\times G)$-sets if and only if $f$ is an isomorphism of $G$-sets.

Now, if $\varphi$ also satisfies the conditions $\underline{\beta'}\varphi=\underline{\beta}$ and $\overline{\alpha'}\varphi=\overline{\alpha}$, then we have
\begin{displaymath}
\beta'f(u)=\beta(u)\quad\textrm{and}\quad \alpha'f(u)=t(u)\alpha (u).
\end{displaymath}
Therefore, the classes of the spans 
\begin{displaymath}
\xymatrix{Y& U\ar[l]_\beta\ar[r]^\alpha & X }\quad\xymatrix{Y& U'\ar[l]_{\beta'}\ar[r]^{\alpha'} & X }
\end{displaymath}
map to the same arrow $\langle G\times U,\, \underline{\beta},\, \overline{\alpha}\rangle$, if and only if there exist an isomorphism  of $G$-sets $f:U\rightarrow U'$ and a morphism of $G$-sets $t:U\rightarrow G^c$, that satisfy
\begin{displaymath}
\beta'f=\beta \quad \textrm{and}\quad \alpha'f=t\cdot\alpha.
\end{displaymath}
Then for example, for a $G$-set $U$ and a morphism $t:U\rightarrow G^c$, the two spans
\begin{displaymath}
\xymatrix{U& U\ar[l]_{id_U}\ar[r]^{id_U} & U }\quad\xymatrix{U& U\ar[l]_{id_U}\ar[r]^{t\cdot id_U} & U }
\end{displaymath} 
map to the identity arrow $\langle G\times U,\, \underline{id_U},\, \overline{id_U}\rangle$. Hence $A$ is not faithful. \textcolor{black}{The relation on the image of spans also shows that by pre-composing an $R$-linear functor $F:\lac_R\rightarrow R$-Mod with functor $A$, we obtain a fused Mackey functor $F\circ A$.}
\end{proof}

\begin{rem}
\textcolor{black}{We notice that n}ot every fused Mackey functor can be obtained in this way, in view of the next theorem and Remark 2.14 in \cite{bham}.
\end{rem}

For biset functors we have the following result.

\begin{teo}
For $R$ a commutative ring with unit we have
\begin{displaymath}
Fun_R(\Omega_R,\, R\textrm{-Mod})\cong Fun_R(\lac_R,\, R\textrm{-Mod}).
\end{displaymath}
\end{teo}
\begin{proof}
Since we have an $R$-linear functor $F:\Omega_R\rightarrow \lac_R$, which is also faithful and injective in objects, we have a functor
\begin{displaymath}
\Phi :Fun_R(\lac_R,\, R\textrm{-Mod})\longrightarrow Fun_R(\Omega_R,\, R\textrm{-Mod}),
\end{displaymath}
defined by pre-composition with $F$.

Given $\zeta :T_1\rightarrow T_2$, a natural transformation in $Fun_R(\lac_R,\, R\textrm{-Mod})$, we have that $\Phi (\zeta):T_1F\rightarrow T_2F$ is a natural transformation  defined in a group $G$ as
\begin{displaymath}
\Phi (\zeta)_G=\zeta_{\frac{\{\bullet\}}{G}}:T_1\left(\frac{\{\bullet\}}{G}\right)\rightarrow T_2\left(\frac{\{\bullet\}}{G}\right).
\end{displaymath}
That $\Phi$ is a full and faithful functor comes then from the fact that an $R$-linear functor between  additive $R$-linear categories preserves coproducts. A proof of this fact can be found in Chapter VIII of \cite{maclane}. So, to see that $\Phi$ is an equivalence of categories, it remains to see that it is dense. But, as seen in Theorem \ref{comp}, the category $\lac_R$ is equivalent to the additive completion of $\Omega_R$, so if we have a biset functor $M$, that is, an $R$-linear functor $M:\Omega_R\rightarrow R$-Mod, then there exists a unique functor $\tilde{M}:\lac_R\rightarrow R$-Mod such that $\tilde{M}F\cong M$.
\end{proof}

\begin{ejem}
With the help of Remark \ref{hom}, we see that the extension of the Burnside functor $RB:\Omega_R\rightarrow R$-Mod to $\lac_R$ is the functor $RB^{\,\_\,,\, \{1\}}(\,\_\,,\, \{\bullet\})$, which maps an object $X/G$ in $\lac_R$ to $RB(X)$, as defined in Section 1.
\end{ejem}

\section*{Acknowledgements}
During the writing  of this paper, Alberto G. Raggi-C\'ardenas and Nadia Romero were partially supported by the 2015 UC MEXUS-CONACYT Collaborative Research Grant ``Representation rings of finite groups''. Nadia Romero was also partially supported by SEP-PRODEP, project PTC-486.

The authors thank Serge Bouc and Radu Stancu, for their most valuable comments and suggestions.

{
\centerline{\rule{5ex}{.1ex}}
\begin{flushleft}
Jes\'us Ibarra, {\tt jtitacho84@gmail.com},\\
 Alberto G. Raggi-C\'ardenas, {\tt graggi@matmor.unam.mx}.\\
Centro de ciencias matem\'aticas, UNAM.\\
 Apartado postal 61-3 (Xangari), 58089, Morelia, Mich., Mexico. 
 \end{flushleft}
 \begin{flushleft}
Nadia Romero, {\tt nadia.romero@ugto.mx}.\\ 
Departamento de matem\'aticas, Universidad de Guanajuato.\\
Jalisco s/n, Mineral de Valenciana, 36240, Guanajuato, Gto., Mexico.\\

\end{flushleft}
}

\bibliographystyle{plain}
\bibliography{tacho}

\end{document}